\documentclass[final]{IEEEtran}

\usepackage{epsfig} 
\usepackage{mathptmx} 
\usepackage{times} 
\usepackage{amsmath} 
\usepackage{amssymb}  
\usepackage{amsthm}

\usepackage{calc}
\usepackage{pdfsync}

\usepackage[usenames,dvipsnames]{xcolor}

\usepackage{tikz}
\usetikzlibrary{arrows}
\usetikzlibrary{graphs,decorations.pathmorphing,decorations.markings}
\usetikzlibrary{calc}

\pgfmathsetmacro\weight{1/2}
\pgfmathsetmacro\third{1/3}
\pgfmathsetmacro\twothirds{2/3}

\tikzset{degil/.style={
            decoration={markings,
            mark= at position 0.5 with {
                  \node[transform shape] (tempnode) {$/$};
                  }
              },
              postaction={decorate}
}
} 

\usepackage{hyperref}   
\hypersetup{
    colorlinks=true,
    linkcolor=blue,
    filecolor=magenta,      
    urlcolor=blue,
        citecolor=red,
    pdftitle={Sharelatex Example},
    pdfpagemode=FullScreen,
}

\usepackage{enumitem}
\setlist[enumerate]{topsep=5pt,itemsep=-0.4ex,leftmargin=8mm}
\setlist[itemize]{leftmargin=*}

\newtheorem{theorem}{Theorem}[section]
\newtheorem{lemma}[theorem]{Lemma}
\newtheorem{proposition}[theorem]{Proposition}
\newtheorem{corollary}[theorem]{Corollary}

\theoremstyle{definition}
\newtheorem{definition}[theorem]{Definition}

\newtheorem{remark}[theorem]{Remark}

\newcounter{syscounter}
\newenvironment{sysnum}{\begin{list}{($\Sigma{\arabic{syscounter}}$)}%
{\settowidth{\labelwidth}{($\Sigma4$)}
\settowidth{\leftmargin}{($\Sigma4$)~}%
\usecounter{syscounter}}}
{\end{list}}

\newcommand \N   {\mathbb{N}}
\newcommand \R   {\mathbb{R}}

\newcommand \A   {\mathcal{A}}
\newcommand \B   {\mathcal{B}}

\newcommand \K   {\mathcal{K}}
\newcommand \Kinf{\mathcal{K_\infty}}

\newcommand \KL  {\mathcal{KL}}
\newcommand \LL  {\mathcal{L}}

\newcommand \Uc   {\mathcal{U}}

\newcommand \srs   {\ \ \Rightarrow\ \ }
\newcommand \srsd   {\ \Rightarrow\ }

\newcommand \Iff   {\Leftrightarrow}

\newcommand \eps {\varepsilon}

\newtheorem*{Sol*}{Solution} 

\newif\ifAndo

\Andotrue

\begin{document}
%
%
%
%
%
%
%
%
%
%

\title{
Criteria for input-to-state practical stability
}
\author{Andrii Mironchenko
\thanks{
This research has been supported by the DFG grant \href{http://www.fim.uni-passau.de/dynamische-systeme/forschung/input-to-state-stability-and-stabilization-of-distributed-parameter-systems/}{"Input-to-state stability and stabilization of distributed parameter systems"} (Wi1458/13-1).
}
\thanks{A. Mironchenko is with 
the Faculty of Computer Science and Mathematics, University of Passau,
94030 Passau, Germany.
Email: andrii.mironchenko@uni-passau.de. Corresponding author.
}
}

\maketitle

{
\begin{abstract}
For a broad class of infinite-dimensional systems, we characterize input-to-state practical stability (ISpS) using the uniform limit property and in terms of input-to-state stability. We specialize our results to the systems with Lipschitz continuous flows and evolution equations in Banach spaces. Even for the special case of ordinary differential equations our results are novel and improve existing criteria for ISpS.
\end{abstract}
}


\section{Introduction}
\label{sec:introduction}

The concept of input-to-state stability (ISS), introduced in \cite{Son89}, has become indispensable for various branches of nonlinear control theory, such as robust stabilization of nonlinear systems \cite{FrK08}, design of nonlinear observers \cite{ArK01}, analysis of large-scale networks \cite{JTP94, DRW07}, etc.

However, in many cases it is impossible (as in quantized control) or too costly to construct a feedback, ensuring ISS behavior of the closed loop system. To address such applications, a relaxation of the ISS concept has been proposed in \cite{JTP94}, called input-to-state practical stability (ISpS, practical ISS). This concept is extremely useful for stabilization of stochastic control systems \cite{ZhX13}, { control under quantization errors \cite{ShL12, JLR09}}, sample-data control \cite{NKK15}, study of interconnections of nonlinear systems by means of small-gain theorems \cite{JTP94,JMW96}, etc. Practical ISS extends the earlier concept of practical asymptotic stability of dynamical systems \cite{LLM90}. 
{ In some works practical ISS is called ISS with bias, see e.g. \cite{JLR09,JLR11}.}

Criteria for ISS in terms of other stability properties are among foundational theoretical results in ISS of ordinary differential equations (ODEs). In \cite{SoW95} Sontag and Wang have shown
that ISS is equivalent to the existence of a smooth ISS Lyapunov function and in
\cite{SoW96} the same authors proved an ISS superposition
theorem, saying that ISS is equivalent to the limit property combined with a local stability.
{
Characterizations of ISS greatly simplify the proofs of other important results, such as small-gain theorems for ODEs \cite{DRW07} and hybrid systems \cite{CaT09, DaK13}, Lyapunov-Razumikhin theory for time-delay systems \cite{Tee98}, \cite{DKM12},
 non-coercive ISS Lyapunov theorems \cite{MiW17b}, relations between ISS and nonlinear $L_2 \to L_2$ stability \cite{Son98}, to name a few examples.
}

Characterizations of ISS for ODEs in \cite{SoW96} heavily exploit the topological structure of an underlying state space $\R^n$, as well as a special type of dynamics (ODEs). Trying to generalize these criteria to infinite-dimensional systems, we face fundamental difficulties: closed bounded balls are never compact in infinite-dimensional normed linear spaces, nonuniformly globally asymptotically stable nonlinear systems do not necessarily have bounded reachability sets, and even if they do, this still does not guarantee uniform global stability \cite{MiW17b}. 
These difficulties have been overcome in a recent work \cite{MiW17b}, where characterizations of ISS have been developed for a general class of control systems, encompassing evolution PDEs, differential equations in Banach spaces, time-delay systems, switched systems, ODEs, etc. 
Characterizations of local ISS of infinite-dimensional systems were obtained slightly earlier in \cite{Mir16}. 
The results in \cite{MiW17b} naturally extend criteria for ISS of ODEs developed in \cite{SoW96}.
New notions and results obtained in \cite{MiW17b} establish a solid background for a solution of further problems. The concept of a uniform limit has been useful in the theory of non-coercive Lyapunov functions \cite{MiW17b, MiW17a, Mir17a} and for characterization of practical uniform asymptotic stability \cite{Mir17a}.

{
Despite a great importance of practical ISS for nonlinear control theory, much less is known about characterizations of practical ISS even in ODE setting. 
Sontag and Wang have shown in \cite[Proposition VI.3]{SoW96} that an ODE system is ISpS iff it is compact ISS, i.e., there is a compact 0-invariant set $\A\subset \R^n$  so that the system has a uniform asymptotic gain (UAG) w.r.t. $\A$.
This is an interesting characterization, but UAG itself is a quite strong property and it may be hard to check it.
It would be desirable to obtain criteria for ISS in terms of weaker properties as limit property, which will be as powerful as characterizations of ISS given in \cite{SoW96} for ODEs and in \cite{MiW17b} for general infinite-dimensional systems.

\textit{In this paper we develop such criteria for practical ISS for a broad class of infinite-dimensional systems.} 
The understanding of the nature of practical ISS will be beneficial for the development of quantized and sample data controllers for infinite-dimensional systems and will give further insights into the ISS theory of infinite-dimensional systems, which is currently a hot topic \cite{MiI15b, KaK17a, PrM12, JNP16b, MaP11, Mir16, MiW17b, JSZ17, MKK17}.  

We prove in Section~\ref{sec:Criterion_for_ISpS} that \textit{a nonlinear infinite-dimensional control system $\Sigma$ possessing bounded reachability sets is practically ISS if and only if there is a bounded subset $\A$ of a state space so that $\Sigma$ has a uniform limit property (ULIM) w.r.t. $\A$}. This criterion can be used to prove ISpS of control systems. On the other hand, we show that any ISpS control system has a so-called complete uniform asymptotic gain property (CUAG), which is stronger than uniform asymptotic gain property (UAG) as defined in \cite{SoW96, MiW17b}.

An important difference of this criterion of ISpS to the criteria of ISS proved in \cite{SoW96, MiW17b} is that it does not involve any kind of stability w.r.t. the set $\A$ (which is necessary for ISS), which significantly simplifies verification of the ISpS property.

We base ourselves on machinery developed in \cite{MiW17b} for characterization of ISS of general infinite-dimensional systems, in particular, we use the notion of the uniform limit and results from \cite{MiW17b} related to this property. Additionally, we develop two further technical results which are of independent interest. 

Firstly, \textit{we introduce a CUAG property and show in Proposition~\ref{prop:CUAG_is_UAG_plus_BRS} that a control system possesses this property if and only if it has a UAG property and if its finite time reachability sets are bounded}.

Secondly, using this CUAG characterization we show in Proposition~\ref{prop:ULIM_plus_BRS_implies_UAG} that \textit{if a system has uniform limit property w.r.t. certain bounded set $\A$ of a state space $X$ and if this system has bounded finite time reachability sets, then there is a set $\B \supset \A$ so that $\Sigma$ has a (much stronger than ULIM) CUAG property w.r.t. $\B$}. 
In our proof we construct a family of such sets.

For systems with Lipschitz continuous flows, we improve some of our criteria in Section~\ref{sec:Special_Classes_of_systems} by showing that $\Sigma$ is ISpS if and only if there is a bounded invariant (under uniformly bounded controls) set $\A$ so that $\Sigma$ is ISS w.r.t. $\A$. 
}

Even specialized to ODE systems our results are novel. In Section~\ref{sec:ISpS_ODEs} we show that \textit{an ODE system
 $\Sigma$ is practically ISS $\Iff$ there is a bounded set $\A$ so that $\Sigma$ has a  limit property w.r.t. $\A$ $\Iff$ $\Sigma$ is compact ISS}. This recovers \cite[Proposition VI.3]{SoW96} and considerably strengthens \cite[Lemma I.4]{SoW96}.

\subsection{Notation}

The following notation will be used throughout these notes. Denote $\R_+:=[0,+\infty)$. For an arbitrary set $S$ and
$n\in\N$ the $n$-fold Cartesian product is $S^n:=S\times\ldots\times
S$. 

{
Let $X$ be a normed linear space with a norm $\|\cdot\|$ and let $\A$ be a nonempty set in $X$. For any $x \in X$ we define a distance from $x\in X$ to $\A$ by $\|x\|_{\A}:=\inf_{y\in \A}\|x-y\|$. 
}
Define also $\|\A\|:=\sup_{x\in\A} \|x\|$.
The open ball in a normed linear space $X$ with radius $r$ around $\A \subset X$ is denoted by $B_r(\A):=\{x\in X :\  \|x\|_{\A}<r\}$.
For short, we denote 
$B_r:=B_r(\{0\})$.
Similarly, \mbox{$B_{r,\Uc}:=\{u \in \Uc: \|u\|_\Uc < r\}$}.
The closure of a set $S \subset X$ w.r.t. norm $\|\cdot\|$ is denoted by $\overline{S}$.
{
With a slight abuse of notation we define $\overline{B_0(\A)}:=\A$ and $\overline{B_{0,\Uc}}=\{0\}$.
}

For the formulation of
stability properties the following classes of comparison functions are
useful:
\begin{equation*}
\begin{array}{ll}
{\K} &:= \left\{\gamma:\R_+\rightarrow\R_+\left|\ \gamma\mbox{ is continuous, strictly} \right. \right. \\
&\phantom{aaaaaaaaaaaaaaaaaaa}\left. \mbox{ increasing and } \gamma(0)=0 \right\}, \\
{\K_{\infty}}&:=\left\{\gamma\in\K\left|\ \gamma\mbox{ is unbounded}\right.\right\},\\
{\LL}&:=\left\{\gamma:\R_+\rightarrow\R_+\left|\ \gamma\mbox{ is continuous and strictly}\right.\right.\\
&\phantom{aaaaaaaaaaaaaaaa} \text{decreasing with } \lim\limits_{t\rightarrow\infty}\gamma(t)=0\},\\
{\KL} &:= \left\{\beta:\R_+\times\R_+\rightarrow\R_+\left|\ \beta \mbox{ is continuous,}\right.\right.\\
&\phantom{aaaaaa}\left.\beta(\cdot,t)\in{\K},\ \beta(r,\cdot)\in {\LL},\ \forall t\geq 0,\ \forall r >0\right\}. \\
\end{array}
\end{equation*}

\section{Preliminaries}
\label{sec:prelim}

In this paper, we consider abstract axiomatically defined time-invariant and forward complete systems on the state space $X$ which are subject to a shift-invariant set of
inputs $\Uc$.

\begin{definition}
\label{Steurungssystem}
Consider the triple $\Sigma=(X,\Uc,\phi)$ consisting of 
\begin{enumerate}[label=(\roman*)]
    \item A normed linear space $(X,\|\cdot\|)$, called the {state space}, endowed with the norm $\|\cdot\|$.
    \item A set of input values $U$, which is a nonempty subset of a certain normed linear space.
    \item A {space of inputs} $\Uc \subset \{f:\R_+ \to U\}$, $0 \in\Uc$          
endowed with a norm $\|\cdot\|_{\Uc}$ satisfying two axioms:
                    
\textit{The axiom of shift invariance} states that for all $u \in \Uc$ and all $\tau\geq0$ the time
shift $u(\cdot + \tau)$ belongs to $\Uc$ with \mbox{$\|u\|_\Uc \geq \|u(\cdot + \tau)\|_\Uc$}.

\textit{The axiom of concatenation} is defined by the requirement that for all $u_1,u_2 \in \Uc$ and for all $t>0$ the concatenation of $u_1$ and $u_2$ at time $t$
\begin{equation}
u(\tau) := 
\begin{cases}
u_1(\tau), & \text{ if } \tau \in [0,t], \\ 
u_2(\tau-t),  & \text{ otherwise},
\end{cases}
\label{eq:Composed_Input}
\end{equation}
belongs to $\Uc$.
Furthermore, if $u_2 \equiv 0$, then \mbox{$\|u\|_{\Uc}\leq \|u_1\|_{\Uc}$}.

    \item A transition map $\phi:\R_+ \times X \times \Uc \to X$.
\end{enumerate}
The triple $\Sigma$ is called a (forward complete) control system, if the following properties hold:

\begin{sysnum}
    \item\label{axiom:FC} Forward completeness: for every $(x,u) \in X \times \Uc$ and
          for all $t \geq 0$ the value $\phi(t,x,u) \in X$ is well-defined.
    \item\label{axiom:Identity} The identity property: for every $(x,u) \in X \times \Uc$
          it holds that $\phi(0, x,u)=x$.
    \item Causality: for every $(t,x,u) \in \R_+ \times X \times
          \Uc$, for every $\tilde{u} \in \Uc$, such that $u(s) =
          \tilde{u}(s)$, $s \in [0,t]$ it holds that $\phi(t,x,u) = \phi(t,x,\tilde{u})$.
    \item \label{axiom:Continuity} Continuity: for each $(x,u) \in X \times \Uc$ the map $t \mapsto \phi(t,x,u)$ is continuous.
        \item \label{axiom:Cocycle} The cocycle property: for all $t,h \geq 0$, for all
                  $x \in X$, $u \in \Uc$ we have
$\phi(h,\phi(t,x,u),u(t+\cdot))=\phi(t+h,x,u)$.
\end{sysnum}
\end{definition}

\begin{remark}
\label{rem:Additional_Restriction}
In compare to the paper \cite{MiW17b}, upon which this note is based, we impose here an additional requirement on the space $\Uc$, that the concatenation of any input $u$ with a zero input has the norm which is not larger than $\|u\|_{\Uc}$. This condition is satisfied by most of the "natural" input spaces.
\end{remark}

\begin{definition}
\label{def:s-invariance}
{
Let a control system $\Sigma=(X,\Uc,\phi)$, a real number $s\geq0$ and $\A \subset X$, $\A \neq \emptyset$ be given.}
\begin{itemize}
{ \item $\A$ is called $s$-invariant if $\phi(t,x,u) \in \A$ for all $t\geq 0$, $x\in \A$ and $u\in \overline{B_{s,\Uc}}$.
}
    \item An $s$-invariant set $\A \subset X$ is called robustly $s$-invariant if for every $\eps > 0$ and any $h>0$ there is a $\delta = \delta (\eps,h)>0$, so that 
\begin{eqnarray}
 t\in[0,h],\ \|x\|_{\A} \leq \delta,\ \|u\|_{\Uc} \leq \delta \; \Rightarrow \;  \|\phi(t,x,u)\|_{\A} \leq \eps.
\label{eq:RobInvSet}
\end{eqnarray}
\end{itemize}
\end{definition}

{
\begin{remark}
\label{rem:0-invariance_and_CEP_property} 
If $\A=\{0\}$, then robust 0-invariance of such $\A$ is precisely the continuity of $\phi$ at the trivial equilibrium, see \cite{MiW17b}.
The concept of $s$-invariance seems to be less standard, but it helps a lot to establish the relation between ISS and ISpS in Theorem~\ref{thm:ISpS_equivalent_ISS_wrt_bounded_sets}, due to the fact that it is much easier to show robustness of $s$-invariant sets with $s>0$ than with $s=0$, see Lemma~\ref{lem:RobustInvariance_of_s_invariant_sets}.
\end{remark}

}

The central notion of this paper is:
\begin{definition}
\label{Def:ISpS_wrt_set}
A control system $\Sigma=(X,\Uc,\phi)$ is called {\it (uniformly) input-to-state practically stable
(ISpS) w.r.t. a nonempty set $\A  \subset X$}, if there exist $\beta \in \KL$, $\gamma \in \Kinf$ and $c>0$
such that for all $ x \in X$, $ u\in \Uc$ and $ t\geq 0$ the following holds:
\begin {equation}
\label{isps_sum}
\| \phi(t,x,u) \|_{\A} \leq \beta(\| x \|_{\A},t) + \gamma( \|u\|_{\Uc}) + c.
\end{equation}

If ISpS property w.r.t. $\A$ holds with $c:=0$, then $\Sigma$ is called \textit{input-to-state stable (ISS)} w.r.t $\A$.
\end{definition}

{
\textit{In what follows we always assume that the set, w.r.t. which the stability property is considered (usually denoted by $\A$) is always nonempty.}
}

We are interested in practical ISS w.r.t. bounded subsets of $X$. The following simple result holds:
\begin{proposition}
\label{prop:Invariance_of_ISpS_wrt_bounded_sets}
Let $\Sigma$ be a control system as in Definition~\ref{Steurungssystem}.
If $\Sigma$ is ISpS w.r.t. a certain bounded set $\A_1 \subset X$, then $\Sigma$ is ISpS w.r.t. any bounded subset of $X$.
\end{proposition}

\begin{proof}
  Let $\A_1$ be a bounded subset of $X$ and let $\Sigma=(X,\Uc,\phi)$ be ISpS w.r.t. $\A_1$.
Using a simple inequality
\begin{eqnarray}
\|x\| - \|\A_1\| \leq \|x\|_{\A_1} \leq \|x\| + \|\A_1\|,
\label{eq:Relations_between_norms_1}
\end{eqnarray}    
which holds for all $x\in X$, as well as the fact that $\beta(a+b,t)\leq \beta(2a,t) + \beta(2b,t)$ for all $a,b,t\geq0$, we arrive at
\begin{align*}
\| \phi(t,x,u) \| - \|\A_1\| &\leq \| \phi(t,x,u) \|_{\A_1} \\
&\leq \beta(\| x \|  +\|\A_1\|,t) + \gamma( \|u\|_{\Uc}) + c\\
&\leq \beta(2\| x \|,t) + \beta(2\|\A_1\|,t) + \gamma( \|u\|_{\Uc}) + c.
\end{align*}    
Defining $\tilde c:=\beta(2\|\A_1\|,0) + \|\A_1\| +c$ we see that
\begin{eqnarray}
\label{eq:tmp_ISpS_estimate}
\| \phi(t,x,u) \| \leq \beta(2\| x \|,t) + \gamma( \|u\|_{\Uc}) + \tilde{c}.
\end{eqnarray}    
Using \eqref{eq:tmp_ISpS_estimate} and \eqref{eq:Relations_between_norms_1} once again, we obtain that $\Sigma$ is ISpS w.r.t. any bounded $\A \subset X$.
\end{proof}

Proposition~\ref{prop:Invariance_of_ISpS_wrt_bounded_sets} motivates the following definition:
\begin{definition}
\label{Def:ISpS}
A control system $\Sigma=(X,\Uc,\phi)$ is called \textit{ISpS}, if there is a bounded set $\A\subset X$ so that $\Sigma$ is ISpS w.r.t. $\A$.
\end{definition}



{Our aim is to prove criteria for practical ISS in terms of more basic stability properties. Next we enlist the most important of such notions.
}
\begin{definition}
\label{def:ULIM_UAG_etc}
A control system $\Sigma=(X,\Uc,\phi)$ 
\begin{itemize}
    \item has \textit{bounded reachability sets (BRS)}, if for any $C>0$ and any $\tau>0$ it holds that 
\[
\sup_{\|x\|\leq C,\ \|u\|_{\Uc} \leq C,\ t \in [0,\tau]}\|\phi(t,x,u)\| < \infty.
\]
    \item is called {\it uniformly globally bounded (UGB)}, if there exist a bounded set $\A  \subset X$, functions $\sigma,\gamma \in \Kinf$ and $c>0$ such that for all $ x \in X$, and all $ u \in \Uc$  it holds that
\begin{equation}
\label{UGBAbschaetzung}
\left\| \phi(t,x,u) \right\|_{\A} \leq \sigma(\|x\|_{\A}) +\gamma(\|u\|_{\Uc}) + c \quad \forall t \geq 0.
\end{equation}

    \item has the {\it uniform asymptotic gain (UAG) property w.r.t. $\A \subset X$}, if there
          exists a
          $ \gamma \in \Kinf$ such that for all $ \eps, r >0$ there is a $ \tau=\tau(\eps,r) < \infty$ s.t. for all $u \in \Uc$ and all $x \in \overline{B_r(\A)}$
\begin{equation}    
\label{UAG_Absch}
t \geq \tau \quad \Rightarrow \quad \|\phi(t,x,u)\|_{\A} \leq \eps + \gamma(\|u\|_{\Uc}).
\end{equation}

    \item has the \textit{limit property (LIM) w.r.t. $\A \subset X$} if there is a $\gamma\in\Kinf$:
for all $x\in X$, $u \in \Uc$ and $\eps>0$ there is a $t=t(x,u,\eps)$:
\[
\|\phi(t,x,u)\|_{\A} \leq \eps + \gamma(\|u\|_{\Uc}).
\]

  \item has the \textit{uniform limit property (ULIM) w.r.t. $\A \subset X$}, if there exists $\gamma\in\Kinf$ so that for all 
    $\eps>0$ and all $r>0$ there is a $\tau = \tau(\eps,r)$ s.t. for all $u\in\Uc$:
\begin{eqnarray}
\hspace{-3mm} \|x\|_{\A} \leq r \srsd \exists t\leq \tau(\eps, r):\; \|\phi(t,x,u)\|_{\A} \leq \eps + \gamma(\|u\|_{\Uc}).
\label{eq:ULIM_ISS_section}
\end{eqnarray}

\end{itemize}
\end{definition}

\begin{remark}
\label{rem:FC_and_BRS}
For ODEs forward completeness implies BRS property, see \cite[Proposition 5.1]{LSW96}. For $\infty$-dimensional systems 
this is not always the case (see \cite[Example 2]{MiW17b}).
\end{remark}

Note that trajectories of ULIM systems do not only approach the ball $B_{\gamma(\|u\|_{\Uc})}(\A)$ (as trajectories of LIM systems do), but they do it uniformly. Indeed, the time of approachability $\tau$ depends only on the norm of the state and $\eps$ and does not depend on he state itself.

{
Uniform asymptotic gain property assures that the trajectories possess a uniform convergence rate. However, UAG property per se does not guarantee that the solutions possess any kind of uniform global bounds 
(one can construct examples of control systems, illustrating this fact, using ideas from \cite[Example 2]{MiW17b}).
Since it is often desirable both to have uniform attraction rates as well as uniform bounds on solutions, we introduce (motivated by \cite[Definition 4.1.3]{Gru02b}, where a similar concept with $\gamma=0$ has been employed) a new notion:
\begin{definition}
\label{def:CUAG}
We say that a control system $\Sigma$ satisfies the completely uniform asymptotic gain property (CUAG) w.r.t. $\A \subset X$, if there are $\beta\in\KL$, 
$\gamma \in \Kinf$ and $C>0$ s.t. for all $x\in X$,\ $u\in\Uc$,\ $t\geq 0$ it holds that:
\begin {equation}
\label{CUAG_sum}
\| \phi(t,x,u) \|_{\A} \leq \beta(\| x \|_{\A} + C,t) + \gamma( \|u\|_{\Uc}).
\end{equation}
\end{definition}
In Section~\ref{sec:CUAG_characterization} we show that CUAG is equivalent to a combination of BRS and UAG properties.

}

\section{Characterizations of ISpS}
\label{sec:Criterion_for_ISpS}

In this section we prove the following characterization of ISpS:
\begin{theorem}
\label{thm:RFC_weak_attr_imply_pUGAS}
{ Let $\Sigma$ be a control system as in Definition~\ref{Steurungssystem}.} The following statements are equivalent:
\begin{enumerate}[label=(\roman*)]
    \item\label{enum:ISpS_item1} $\Sigma$ is ISpS
    \item\label{enum:ISpS_item2} {There is a bounded 0-invariant set $\A\subset X$ so that $\Sigma$ is CUAG w.r.t. $\A$.}
    \item\label{enum:ISpS_item3} $\Sigma$ is BRS and  there is a bounded set $\A\subset X$ so that $\Sigma$ is ULIM w.r.t. $\A$.
\end{enumerate}
\end{theorem}

{
Theorem~\ref{thm:RFC_weak_attr_imply_pUGAS} can be used in two ways. On the one hand, to prove ISpS of a system, we can merely check the conditions in item \ref{enum:ISpS_item3} of Theorem~\ref{thm:RFC_weak_attr_imply_pUGAS}.
On the other hand, if we know that a certain system is ISpS, item \ref{enum:ISpS_item2} shows that it enjoys also a CUAG property w.r.t. a certain bounded 0-invariant set.
}

The proof will be divided into several lemmas which will be subdivided into two subsections.
We start with a characterization of the CUAG property.

\subsection{Characterization of CUAG property}
\label{sec:CUAG_characterization}

We start with a preliminary {result}.
%
    %
%
\begin{proposition}
\label{prop:lemmas_from_MiW17b}
Let $\Sigma$ be a control system. If $\Sigma$ is {BRS and} ULIM w.r.t. certain bounded set $\A \subset X$, then $\Sigma$ is UGB.
\end{proposition}

\begin{proof}
This claim was proved for $\A=\{0\}$ in \cite[Proposition 7]{MiW17b}. The proof for general bounded $\A$ is analogous.
\end{proof}

{

Now we provide a simple restatement of the UGB property.
\begin{lemma}
\label{lem:UGB_characterizations}
A control system $\Sigma$ is UGB w.r.t. $\A\subset X$ iff there are $\sigma_1,\gamma\in\Kinf$ and $c>0$ so that
 \begin{equation}
    \label{eq:UGB_alternative_1}
    \| \phi(t,x,u) \|_{\A} \leq \sigma_1(\|x\|_{\A} + c) + \gamma(\|u\|_{\Uc}).
\end{equation}      
\end{lemma}

\begin{proof}
Let $\Sigma$ be UGB. Then there are $\sigma,\gamma\in\Kinf$ and $c>0$ so that for any $x\in X$, any $u\in\Uc$ and any $t\ge 0$ it holds that
\begin{eqnarray*}
 \| \phi(t,x,u) \|_{\A} &\leq& \sigma(\|x\|_{\A}) + \gamma(\|u\|_{\Uc}) + c\\
                                &\leq& \sigma(\|x\|_{\A} + c) + \|x\|_{\A}  + c + \gamma(\|u\|_{\Uc})\\
                                &=:& \sigma_1(\|x\|_{\A} + c) + \gamma(\|u\|_{\Uc})
\end{eqnarray*}
for $\sigma_1(r) := \sigma(r) + r$, $r\ge 0$.
Clearly, $\sigma_1\in\Kinf$ and hence \eqref{eq:UGB_alternative_1} holds.

Conversely, let \eqref{eq:UGB_alternative_1} hold.
Then there are $\sigma_1,\gamma\in\Kinf$ and $c>0$ s.t. for any $x\in X$, any $u\in\Uc$ and any $t\ge 0$ it holds that
 \begin{eqnarray*}
    \| \phi(t,x,u) \|_{\A} &\leq& \sigma_1(\|x\|_{\A}+c) + \gamma(\|u\|_{\Uc})\\
                                        &\leq& \sigma_1(2\|x\|_{\A}) + \gamma(\|u\|_{\Uc}) + \sigma_1(2c),
\end{eqnarray*}    
and thus $\Sigma$ is UGB.  
\end{proof}

The following proposition gives a useful characterization of CUAG.
\begin{proposition}
\label{prop:CUAG_is_UAG_plus_BRS}
Let $\A\subset X$ be a bounded set.
A control system $\Sigma$ is CUAG w.r.t. $\A$ $\Iff$ $\Sigma$ is BRS and UAG w.r.t. $\A$.
\end{proposition}

\begin{proof}
"$\Rightarrow$". Let $\Sigma$ be a CUAG control system. Then $\Sigma$ is BRS. Now for every $\eps>0$ and any $r>0$ pick $\tau(\eps,r)$ as a solution of the equation $\beta(C+r,\tau)=\eps$ (if it exists), or $\tau(\eps,r):=\infty$, if the equation has no solution.
This shows UAG of $\Sigma$ w.r.t. $\A$.

"$\Leftarrow$". Let a control system $\Sigma$ be UAG w.r.t. $\A$ and BRS. According to Proposition~\ref{prop:lemmas_from_MiW17b},
the system $\Sigma$ is UGB and in view of Lemma~\ref{lem:UGB_characterizations}
there are $\sigma_1,\gamma\in\Kinf$ and $c>0$ so that \eqref{eq:UGB_alternative_1} holds
for all $x\in X$, $t\geq0$ and $u\in\Uc$.

Without loss of generality, we may assume that $\gamma$ in \eqref{UAG_Absch} is the same as in \eqref{eq:UGB_alternative_1} (otherwise pick $\gamma$ as a maximum of both of them).

Fix arbitrary $r \in \R_+$ and define $\eps_n:= 2^{-n}  \sigma(r + c)$, for $n \in \N$. The UAG property implies that there exists a sequence of times
$\tau_n:=\tau(\eps_n,r)$, $n\in\N$ which we may without loss of generality assume
to be strictly increasing, such that for all $x \in \overline{B_r(\A)}$ and all $u \in \Uc$
\[
\|\phi(t,x,u)\|_{\A} \leq \eps_n + \gamma(\|u\|_{\Uc})\quad  \forall t \geq \tau_n.
\]
From Proposition~\ref{lem:UGB_characterizations} we see that the previous inequality is valid also for $n=0$, if we set $\tau_0 := 0$.
Define $\tilde\omega(r,\tau_n):=\eps_{n-1} = 2^{-n +1} \sigma(r + c)$, for $n \in \N$, $n \neq 0$ and $\tilde\omega(r,0):=2\eps_0=2\sigma(r+c)$. 
Here we assume $r\in[-c,+\infty)$.

Now extend the definition of $\tilde\omega$ to a function $\tilde\omega(r,\cdot) \in \LL$, $r\in[-c,+\infty)$ (see \cite[Lemma~7]{MiW17b} for details) so that
\begin{eqnarray}
\|\phi(t,x,u)\|_{\A} \leq \tilde\omega(\|x\|_{\A},t) + \gamma(\|u\|_{\Uc}).
\label{eq:tmp_CUAG_estimate}
\end{eqnarray}
Set $\omega(r,t):=\tilde\omega(r-c,t)$, $r\geq 0$, $t\geq0$. 

As in the proof of \cite[Lemma~7]{MiW17b}, there is $\beta\in\KL$ so that $\omega(r,t)\leq \beta(r,t)$ for all $r,t\ge 0$.
From \eqref{eq:tmp_CUAG_estimate} we obtain
\begin{eqnarray*}
\|\phi(t,x,u)\|_{\A} \leq \beta(\|x\|_{\A} + c,t) + \gamma(\|u\|_{\Uc}),
\end{eqnarray*}
which shows CUAG of $\Sigma$.
\end{proof}

{

\begin{proposition}
\label{prop:CUAG_implies_ISpS}
If there is a bounded $\A \subset X$ so that a control system $\Sigma = (X,\Uc,\phi)$ is CUAG w.r.t. $\A$, then $\Sigma$ is ISpS.
\end{proposition}

\begin{proof}
Let $\Sigma=(X,\Uc,\phi)$ be a control system which is CUAG w.r.t. certain bounded $\A\subset X$. Then there are $\beta\in\KL$ and $\gamma\in\Kinf$ so that for all $x\in X,\ u\in\Uc,\ t\geq 0$ we have
\begin {eqnarray*}
\| \phi(t,x,u) \|_{\A} &\leq& \beta(\| x \|_{\A} + C,t) + \gamma( \|u\|_{\Uc})\\
&\leq& \beta(2\| x \|_{\A},t) + \beta(2C,t) + \gamma( \|u\|_{\Uc})\\
&\leq& \beta(2\| x \|_{\A},t) + \beta(2C,0) + \gamma( \|u\|_{\Uc}),
\end{eqnarray*}
which shows ISpS of $\Sigma$.
\end{proof}
}

}

{

\subsection{Improving stability properties via enlarging of the attracting set}
\label{sec:Stab_prop_improvement}

Assume that a control system $\Sigma$ is ULIM w.r.t a certain set $\A\subset X$. Then it is clear that $\Sigma$ is ULIM w.r.t. any $\B\supset \A$.
However, it may exist certain subsets of $X$ w.r.t. which $\Sigma$ has better properties than merely ULIM.
In this section we show that this is indeed the case provided $\Sigma$ is BRS.
We exploit a simple lemma:
\begin{lemma}
\label{lem:Averages_of_increasing_functions} 
Let $f:\R_+\to\R$ be strictly increasing. Then $g:t\mapsto \frac{1}{t}\int_0^t f(s)ds$ is again strictly increasing.
\end{lemma}

Assume that $\Sigma=(X,\Uc,\phi)$ is given. For any $\A\subset X$, any $\eps>0$ and any $\gamma\in\Kinf$ define
\begin{eqnarray}
\A_{\eps,\gamma}:=\{\phi(t,x,u):t\in\R_+,\ x\in B_{\eps}(\A),\ \|u\|_{\Uc}\leq\gamma^{-1}(\tfrac{\eps}{2})\}.
\label{eq:Aeps_def}
\end{eqnarray}
Note that from the identity axiom ($\Sigma$\ref{axiom:Identity}), for each $\eps>0$ and any $\gamma\in\Kinf$ it holds that 
$\A\subset B_\eps(\A) \subset \A_{\eps,\gamma}$.
The construction of the sets $\A_{\eps,\gamma}$ is motivated by the notion of the positive prolongation of a set, see \cite{BhS02}.

Next we show the central technical result in this paper. It may be viewed as a strengthening of \cite[Lemma VI.2]{SoW96}.
}

%
%

\begin{proposition}
\label{prop:ULIM_plus_BRS_implies_UAG}
Assume that $\Sigma$ is a BRS control system and $\Sigma$ has the ULIM property w.r.t. a bounded (not necessarily 0-invariant) set $\A \subset X$, with $\gamma\in\Kinf$ as in \eqref{eq:ULIM_ISS_section}. 
Then for any $\eps>0$ the set $\A_{\eps,\gamma}$ is bounded, 0-invariant and $\Sigma$ is CUAG w.r.t. $\A_{\eps,\gamma}$.
\end{proposition}

\begin{proof}
We divide the proof into four parts.

\textbf{Boundedness of $\A_{\eps,\gamma}$}. Pick any $\eps>0$, any $R>0$ and fix them. Since $\Sigma$ is ULIM w.r.t. $\A$, there is a $\tilde\tau=\tilde\tau(\eps, R)>0$ so that for any $u\in\Uc$
\begin{eqnarray}
\|x\|_{\A} \leq R \srsd \exists \bar{t}\leq \tilde\tau:\ \|\phi(\bar{t},x,u)\|_{\A} \leq \frac{\eps}{2} +\gamma(\|u\|_{\Uc}).
\label{eq:ULIM_UAG_Section0}
\end{eqnarray}
Without loss of generality we can assume that $\tilde\tau$ is decreasing w.r.t. $\eps$ and increasing w.r.t. $R$.
Then $\tilde\tau$ is integrable.
Define
\[
\tau(\eps,R):=\frac{2}{\eps R}\int_R^{2R}\int_{\eps/2}^{\eps} \tilde\tau(\eps_1,R_1)d\eps_1dR_1.
\]
{
Since $\tilde\tau$ is strictly increasing w.r.t. the second argument and strictly decreasing w.r.t. the first one, for any $\eps,R>0$ we have that 
$$\min\{\tilde\tau(\eps_1,R_1):\ \eps_1\in[\eps/2,\eps],\ R_1\in[R,2R] \} = \tilde\tau(\eps,R)$$
and thus $\tau(\eps,R)\geq\tilde\tau(\eps,R)$.

By standard lemmas from analysis, $\tau$ is continuous.

Let us show that $\tau$ is increasing w.r.t. $R$. Pick any $R_1,R_2>0$ so that $R_2>R_1$. We have for $i=1,2$:
\[
\tau(\eps,R_i) = \frac{2}{\eps}\int_{\eps/2}^{\eps} \frac{1}{R_i}\int_{R_i}^{2R_i} \tilde\tau(\eps_1,R) dR d\eps_1.
\]
Since $\tilde\tau(\eps_1,\cdot)$ is strictly increasing, and since $R_2 >R_1$, we have:
\[
\frac{1}{R_2}\int_{R_2}^{2R_2} \tilde\tau(\eps_1,R) dR > \frac{1}{R_2}\int_{R_1}^{R_1+R_2} \tilde\tau(\eps_1,R) dR
\]
Again since $\tilde\tau(\eps_1,\cdot)$ is strictly increasing and applying Lemma~\ref{lem:Averages_of_increasing_functions} we see that
\[
\frac{1}{R_2}\int_{R_1}^{R_1+R_2} \tilde\tau(\eps_1,R) dR > \frac{1}{R_1}\int_{R_1}^{2R_1} \tilde\tau(\eps_1,R) dR.
\]
This shows that $\tau(\eps,R_2)>\tau(\eps,R_1)$, and thus $\tau$ is increasing w.r.t. the second argument.

Analogously, $\tau$ is decreasing w.r.t. $\eps$. 
}
Consequently, for any $u\in\Uc$
\begin{eqnarray}
\|x\|_{\A} \leq R \srsd \exists \bar{t}\leq \tau(\eps,R):\ \|\phi(\bar{t},x,u)\|_{\A} \leq \frac{\eps}{2} +\gamma(\|u\|_{\Uc})
\label{eq:ULIM_UAG_Section1}
\end{eqnarray}
and in particular
\begin{equation*} 
\begin{split}
&\|x\|_{\A} \leq \frac{\eps}{2} +  \gamma(\|u\|_{\Uc}) \\
& \ \Rightarrow \ \exists \tilde{t} {\in} (0, \tau(\eps, \frac{\eps}{2} {+}\gamma(\|u\|_{\Uc}))):\ \|\phi(\tilde{t},x,u)\|_{\A} \leq \frac{\eps}{2} {+}\gamma(\|u\|_{\Uc}).
\end{split}
\end{equation*}
This means, that trajectories corresponding to the input $u$, emanating from $B_{\frac{\eps}{2} +\gamma(\|u\|_{\Uc})}(\A)$ return to this ball in time not larger than $\tau(\eps, \frac{\eps}{2} +\gamma(\|u\|_{\Uc}))$.
In particular, for any $\eps>0$, any $x\in B_{\eps}(\A)$ and any $u \in \Uc$: $\|u\|_{\Uc} \leq \gamma^{-1}(\tfrac{\eps}{2})$ there is $\bar{t}\leq \tau(\eps,\eps)$
so that $\phi(\bar{t},x,u) \in B_{\eps}(\A)$.

For any $t>0$ due to the cocycle property it holds that
\[
\phi(t+\bar{t},x,u)= \phi(t,\phi(\bar{t},x,u),u(\bar{t} + \cdot)).
\]
The axiom of shift invariance tells us that $\|u(\bar{t} + \cdot)\|_{\Uc}\leq \|u\|_{\Uc} \leq \gamma^{-1}(\tfrac{\eps}{2})$.
Hence $\phi(t+\bar{t},x,u) \in \A_{\eps,\gamma}$ for all $t\geq0$ by definition of $\A_{\eps,\gamma}$ and for our system $\Sigma$ the set $\A_{\eps,\gamma}$ can be represented as
\[
\A_{\eps,\gamma} = \{\phi(t,x,u): t\in[0,\tau(\eps,\eps)],\ x\in B_{\eps}(\A),\ \|u\|_{\Uc}\leq\gamma^{-1}(\tfrac{\eps}{2})\}.
\]
Since $\Sigma$ is BRS, $\A_{\eps,\gamma}$ is bounded for any $\eps>0$.

\textbf{0-invariance of $\A_{\eps,\gamma}$}.
Let $y\in\A_{\eps,\gamma}$. Then there are certain $s\geq0$, $x \in B_\eps(\A)$ and $u\in \Uc$: $\|u\|_{\Uc} \leq \gamma^{-1}(\tfrac{\eps}{2})$ so that $y=\phi(s,x,u)$.
Due to cocycle property it holds for any $t\geq0$ that 
\begin{eqnarray*}
\phi(t,y,0) = \phi(t+s,x,w),\quad
w(\tau) := 
\begin{cases}
u(\tau), & \text{ if } \tau \in [0,s], \\ 
0,  & \text{ otherwise}.
\end{cases}
\end{eqnarray*}
By the axiom of concatenation we have $w\in\Uc$ $\|w\|_\Uc \leq \|u\|_\Uc \leq \gamma^{-1}(\tfrac{\eps}{2})$ and hence $\phi(t,y,0) \in \A_{\eps,\gamma}$.

\textbf{$\Sigma$ is UAG w.r.t. $\A_{\eps,\gamma}$}.
Let us fix $\eps>0$ and $R>0$ and revisit the implication \eqref{eq:ULIM_UAG_Section1}, which implies that
\begin{eqnarray}
\|x\|_{\A} \leq R,\; \|u\|_{\Uc}\leq  \gamma^{-1}(\tfrac{\eps}{2}) \srsd \exists \bar{t}\leq \tau:\; \|\phi(\bar{t},x,u)\|_{\A} \leq \eps
\label{eq:ULIM_UAG_Section3}
\end{eqnarray}
and
\begin{equation}
\begin{split}
\|x\|_{\A} \leq R,&\ \|u\|_{\Uc}\geq  \gamma^{-1}(\tfrac{\eps}{2}) \\
&\srs  \exists \bar{t}\leq \tau:\ \|\phi(\bar{t},x,u)\|_{\A} \leq 2\gamma(\|u\|_{\Uc}).
\end{split}
\label{eq:ULIM_UAG_Section4}
\end{equation}
{
Taking \eqref{eq:ULIM_UAG_Section3} and \eqref{eq:ULIM_UAG_Section4} and using the reasoning exploited to show boundedness of $\A_{\eps,\gamma}$ we obtain
\begin{eqnarray}
\|x\|_{\A} \leq R,\; \|u\|_{\Uc}\leq  \gamma^{-1}(\tfrac{\eps}{2}),\; t\geq\tau \srsd \phi(t,x,u)\in \A_{\eps,\gamma}
\label{eq:ULIM_UAG_Section3_a}
\end{eqnarray}
and
\begin{equation}
\begin{split}
\|x\|_{\A} \leq R,&\ \|u\|_{\Uc}\geq  \gamma^{-1}(\tfrac{\eps}{2}),\; t\geq\tau \srsd \phi(t,x,u)\in \A_{2\gamma(\|u\|_{\Uc}),\gamma}
\end{split}
\label{eq:ULIM_UAG_Section4_a}
\end{equation}
Thus, for all $u\in\Uc$ we have
\begin{eqnarray}
\|x\|_{\A} \leq R,\ t \geq \tau(R,\eps)  \srsd \phi(t,x,u) \in \A_{\eps,\gamma} \cup  \A_{2\gamma(\|u\|_{\Uc}),\gamma}.
\label{eq:ULIM_UAG_Section5}
\end{eqnarray}
}
As $\A_{k,\gamma}$ is bounded for any $k>0$, the  following function is well-defined:
\begin{eqnarray}
f_\eps: s \mapsto \sup_{x\in \A_{s,\gamma}}\|x\|_{\A_{\eps,\gamma}}, \quad s \geq 0.
\label{eq:f_eps_def}
\end{eqnarray}
Since $\A_{k_1,\gamma} \subset \A_{k_2,\gamma}$ for $k_1<k_2$, $f_{\eps}$ is nondecreasing and $f_{\eps}(s)=0$ for $s\in[0,\eps]$. Hence there exists $\sigma_\eps \in\Kinf$:
$f_\eps(s) \leq \sigma_\eps (s)$ for all $s\geq0$.

With this notation we can reformulate \eqref{eq:ULIM_UAG_Section5} into
\begin{equation}
\begin{split}
\|x\|_{\A} \leq R,\ &u\in\Uc,\ t \geq \tau(R,\eps) \\
& \srsd \|\phi(t,x,u)\|_{\A_{\eps,\gamma}} \leq \sigma_\eps(2 \gamma(\|u\|_{\Uc})).
\label{eq:ULIM_UAG_Section6}
\end{split}
\end{equation}
Since $\A \subset \A_{\eps,\gamma}$ for any $\eps>0$ it follows that 
\mbox{$\|x\|_{\A} \leq \|x\|_{\A_{\eps,\gamma}} + C$}, where

\begin{center}
$C:=\sup_{y\in\partial\A_{\eps,\gamma},\ z\in\partial\A,\ [y,z]\bigcap\A=\emptyset}\|y-z\|$ 
\end{center}
and
\mbox{$[y,z]:=\{ty+(1-t)z:t\in(0,1)\}$}.

Define $\tilde\tau(R,\eps):=\tau(R+C,\eps)$. We have:
\begin{equation}
\begin{split}
\|x\|_{\A_{\eps,\gamma}} \leq R,\ u&\in\Uc,\ t \geq \tilde\tau(R,\eps) \\
& \srsd \|\phi(t,x,u)\|_{\A_{\eps,\gamma}} \leq \sigma_\eps(2 \gamma(\|u\|_{\Uc})).
\label{eq:ULIM_UAG_Section7}
\end{split}
\end{equation}
This shows that $\Sigma$ is UAG w.r.t. $\A_{\eps,\gamma}$.

{
Note that here $\eps$ is a design parameter of a set $\A_{\eps,\gamma}$ w.r.t. which $\Sigma$ is UAG, and it is not connected to the parameter $\eps$ in the Definition~\ref{def:ULIM_UAG_etc} of the UAG property.
}

\textbf{$\Sigma$ is CUAG w.r.t. $\A_{\eps,\gamma}$}. Follows by Proposition~\ref{prop:CUAG_is_UAG_plus_BRS}.
\end{proof}

\begin{remark}
\label{rem:Minimal_UAG_sets}
Proposition~\ref{prop:ULIM_plus_BRS_implies_UAG} shows under certain assumptions that $\Sigma$ is UAG w.r.t. $\A_{\eps,\gamma}$ for any $\eps>0$. It is natural to ask, what is the smallest set w.r.t. which $\Sigma$ is UAG, in particular, 
whether $\Sigma$ is UAG w.r.t. $\bigcap_{\eps>0}\A_{\eps,\gamma}$. We do not follow this line here, but a reader may consult \cite[Section 3.2]{Mir17a} for related results.
{
In order to understand additional difficulties which arise on this way, note that if $\varepsilon_1<\varepsilon_2$, then $f_{\varepsilon_1}(s) \geq f_{\varepsilon_2}(s)$ for all $s\in\R_+$ (this follows from \eqref{eq:f_eps_def}). However, this does not imply that there is a continuous function $f_0$ with $f_0(0)=0$ so that $f_0(s)\geq f_\varepsilon(s)$ for any $s\geq 0$ and any $\varepsilon>0$.
}
\end{remark}

Finally, we can prove the main result of this paper:
\begin{proof}\textbf{(of Theorem~\ref{thm:RFC_weak_attr_imply_pUGAS})}

\ref{enum:ISpS_item1} $\Rightarrow$ \ref{enum:ISpS_item2}. 
Assume, that $\Sigma$ is an ISpS control system.
Then there are $\beta\in\KL$, $\gamma\in\Kinf$ and $c>0$ so that for all $x\in X$, $t\geq 0$ and $u\in\Uc$ we have
\begin{equation}
\label{eq:practical_UGAS_Xnorm}
\hspace{-1mm}\|\phi(t,x,u)\| \leq \beta(\|x\|,t) +\gamma(\|u\|_{\Uc}) + c.
\end{equation}
{
For any $y\in X$ and any $c>0$ it holds that $\|y\|_{\overline{B_c(0)}}= \max\{\|y\| - c,0\}$ and $\|y\|\leq \|y\|_{\overline{B_c(0)}} + c$ and we infer from 
\eqref{eq:practical_UGAS_Xnorm} for $\phi(t,x,u)\notin \overline{B_c(0)}$ that
\begin{equation}
\label{eq:practical_UGAS_Xnorm_Bc0}
\hspace{-1mm}\|\phi(t,x,u)\|_{\overline{B_c(0)}} \leq \beta(\|x\|_{\overline{B_c(0)}} +c,t) +\gamma(\|u\|_{\Uc}).
\end{equation}
Otherwise, if $\phi(t,x,u)\in \overline{B_c(0)}$, then $\|\phi(t,x,u)\|_{\overline{B_c(0)}}=0$ and \eqref{eq:practical_UGAS_Xnorm_Bc0} also holds.
Thus $\Sigma$ is CUAG w.r.t. $\overline{B_c(0)}$ (however, $\overline{B_c(0)}$ does not have to be 0-invariant).
According to Proposition~\ref{prop:ULIM_plus_BRS_implies_UAG} there is a bounded 0-invariant set $\A$ so that $\Sigma$ is CUAG w.r.t. $\A$.

\ref{enum:ISpS_item2} $\Rightarrow$ \ref{enum:ISpS_item3}.  Clear.

\ref{enum:ISpS_item3} $\Rightarrow$ \ref{enum:ISpS_item1}. Proposition~\ref{prop:ULIM_plus_BRS_implies_UAG} implies that there is a bounded set $\A\subset X$ so that $\Sigma$ is CUAG w.r.t. $\A$. Proposition~\ref{prop:CUAG_implies_ISpS} shows ISpS of $\Sigma$.
}
%
%
%
\end{proof}

%

\section{Special classes of systems}
\label{sec:Special_Classes_of_systems}

{
One of the criteria for ISpS, shown in Theorem~\ref{thm:RFC_weak_attr_imply_pUGAS}, states that ISpS of a control system $\Sigma$ is equivalent to existence of a 0-invariant set $\A$ so that $\Sigma$ has a CUAG property w.r.t. $\A$. It is natural to ask whether an ISS property (which is stronger than CUAG) holds w.r.t. this set. 
This problem can be approached using the following result:
\begin{proposition}
\label{prop:ISS_wrt_A_criterion}
Let $\A \subset X$ be bounded. A control system $\Sigma$ is ISS w.r.t. $\A$ iff $\Sigma$ is CUAG w.r.t. $\A$ and $\A$ is a robustly 0-invariant set.
\end{proposition}

\begin{proof}
This result has been shown in \cite[Theorem 5]{MiW17b} for $\A=\{0\}$. 
{
In view of a characterization of the CUAG property in Proposition~\ref{prop:CUAG_is_UAG_plus_BRS}.
}
The proof for general bounded sets $\A$ is completely analogous and hence is omitted.
\end{proof}

However, the proof or disproof of robust 0-invariance of the constructed 0-invariant set $\A$ is in general a difficult task. 
In Section~\ref{sec:Lipschitz_flows} we show that this is possible for systems with a Lipschitz continuous flow under certain restrictions on the input space. The notion of $s$-invariance and the corresponding Lemma~\ref{lem:RobustInvariance_of_s_invariant_sets} help us heavily on this way.
Next, in Section~\ref{sec:Semilinear_equations} above criteria will be applied to semilinear evolution equations in Banach spaces.
Finally, Section~\ref{sec:ISpS_ODEs} is devoted to ODE systems. In this case we can strengthen the results even further, due to the fact that ULIM and LIM notions coincide for ODE systems.
}

\subsection{Systems with Lipschitz continuous flows}
\label{sec:Lipschitz_flows}

\begin{definition}
\label{def:Lipschitz}
{ The flow of a control system $\Sigma$} is called Lipschitz continuous on compact intervals (for uniformly bounded inputs), if 
for any $h>0$ and any $r>0$ there is $L>0$ s.t. for any $x,y \in \overline{B_r}$,
all $u \in \overline{B_{r,\Uc}}$ and all $t \in [0,h]$ it holds that 
\begin{eqnarray}
\|\phi(t,x,u) - \phi(t,y,u) \| \leq L \|x-y\|.
\label{eq:Flow_is_Lipschitz}
\end{eqnarray}    
\end{definition}
Many classes of systems possess flows which are Lipschitz continuous on compact intervals. Semilinear evolution equations and ODEs with Lipschitz continuous nonlinearities, which are considered next, are particular of examples of such kind of systems. Lipschitz continuity of the flow helps to prove such significant results as e.g. converse Lyapunov theorems for infinite-dimensional systems \cite[Section 3.4]{KaJ11b}, \cite{MiW17a}.

\begin{lemma}
\label{lem:RobustInvariance_of_s_invariant_sets}
Let a control system $\Sigma$ be given and let $\A \subset X$ be a bounded $s$-invariant set, for a certain $s>0$. If the flow of $\Sigma$ is Lipschitz continuous on compact intervals, then $\A$ is a robustly $s$-invariant set.
\end{lemma}

\begin{proof}
Let $\A$ be a bounded $s$-invariant set, for a certain $s>0$. 

Pick any $\eps>0$, $h>0$ and set $r := \|\A\|+1$. For this $r$ there is a $L=L(2r,h)$ so that for any $x,y \in \overline{B_{2r}}$, $u \in \overline{B_{2r,\Uc}}$ and $t \in [0,h]$ it holds that
\begin{eqnarray*}
\|\phi(t,x,u) -\phi(t,y,u)\| \leq L \|x-y\|.
\end{eqnarray*}
Set $\delta:= \min\big\{\frac{\eps}{L},r\big\}$ and pick any $x\in \overline{B_\delta(\A)}$ (hence $\|x\|\leq \|\A\|+\delta < 2r$). 
Then there is a $y \in \A$: $\|x-y\| \leq \delta$ and the following estimates hold for $t\in[0,h]$ and 
$u \in \overline{B_{\min\{r,s\},\Uc}}$ (note that $\phi(t,y,u) \in\A$ due to $s$-invariance of $\A$):
\begin{align*}
\|\phi(t,x,u)\|_{\A} = \inf_{z\in \A}\|\phi(t,x,u) -z\| \leq& ~\|\phi(t,x,u) - \phi(t,y,u)\|\\
\leq& L \|x-y\|\\
\leq& \eps.
\end{align*}

This shows robust invariance of $\A$.
\end{proof}

{
\begin{lemma}
\label{lem:nonzero_Invariance_Aeps}
Let $\Sigma$ be a control system. Let also for each $u,v \in\Uc$ and for all $t>0$ the concatenation
\begin{center}
$w(s) := 
\begin{cases}
u(s), & \text{ if } s \in [0,t], \\ 
v(s-t),  & \text{ otherwise},
\end{cases}
$ 
\end{center}
of $u$ and $v$ at time $t$ 
satisfy the property
\begin{eqnarray}
\|w\|_{\Uc} \leq \max\{\|u\|_{\Uc}, \|v\|_{\Uc}\}.
\label{eq:Norm_of_concatenation_upper_estimate}
\end{eqnarray}
Then for any $\eps>0$ and any $\gamma\in\Kinf$ the space $\A_{\eps,\gamma}$ is $\gamma^{-1}(\tfrac{\eps}{2})$-invariant.
\end{lemma}
}

\begin{proof}
Pick any $y\in\A_{\eps,\gamma}$ and any $v\in \overline{B_{\gamma^{-1}(\frac{\eps}{2}),\Uc}}$. 
Then there are certain $s\geq 0$, $x \in B_\eps(\A)$ and $u\in  \overline{B_{\gamma^{-1}(\frac{\eps}{2}),\Uc}}$ so that $y=\phi(s,x,u)$.
Due to cocycle property it holds for each $t>0$ that
\begin{eqnarray}
\phi(t,y,v) = \phi(t+s,x,w),
\label{eq:cocycle_property_invariance_proof}
\end{eqnarray}
where $w$ is a concatenation of $u$ and $v$ at time $s$.

In view of the assumption \eqref{eq:Norm_of_concatenation_upper_estimate} it holds that 
$\|w\|_{\Uc}\leq \gamma^{-1}(\tfrac{\eps}{2})$, and hence $\phi(t+s,x,w) \in \A_{\eps,\gamma}$ due to the definition of $\A_{\eps,\gamma}$.
Hence $\phi(t,y,v) \in \A_{\eps,\gamma}$.
\end{proof}

\begin{remark}
\label{rem:Restrictions_on_the_input_space}
An additional assumption on the input space $\Uc$ in Lemma~\ref{lem:nonzero_Invariance_Aeps} restricts the class of inputs. In particular, the inputs from $L_p$ spaces or from Sobolev spaces do not fulfill it. However, the spaces of continuous, piecewise continuous and $L_\infty$ functions do satisfy it (w.r.t. the natural $\sup$ or esssup-norm respectively).
\end{remark}

Finally, we are able to characterize ISpS in terms of ISS:
\begin{theorem}
\label{thm:ISpS_equivalent_ISS_wrt_bounded_sets}
Let $\Sigma=(X,\Uc,\phi)$ be a control system, $\phi$ be Lipschitz continuous on compact time intervals and the input space $\Uc$ satisfy the assumptions of Lemma~\ref{lem:nonzero_Invariance_Aeps}.
Then:
\begin{center}
$\Sigma$ is ISpS $\quad\Iff\quad$ For any $s>0$ there is a bounded \\
\hfill $s$-invariant set $\A\subset X$: $\Sigma$ is ISS w.r.t. $\A$
\end{center}
%
\end{theorem}

\begin{proof}

{
"$\Leftarrow$". Since $\Sigma$ is CUAG w.r.t. $\A$, Proposition~\ref{prop:CUAG_implies_ISpS} shows that $\Sigma$ is ISpS.

"$\Rightarrow$". According to Theorem~\ref{thm:RFC_weak_attr_imply_pUGAS} { and Proposition~\ref{prop:ULIM_plus_BRS_implies_UAG}}, ISpS of $\Sigma$ with a corresponding gain $\gamma\in\Kinf$ implies that for each $\eps>0$ the system $\Sigma$ is CUAG w.r.t. $\A_{\eps,\gamma}$.
}
 Since the assumptions of Lemma~\ref{lem:nonzero_Invariance_Aeps} hold, 
$\A_{\eps,\gamma}$ is a $\gamma^{-1}(\tfrac{\eps}{2})$-invariant bounded set. 
Now Lemma~\ref{lem:RobustInvariance_of_s_invariant_sets} shows that $\A_{\eps,\gamma}$ is a robust $\gamma^{-1}(\tfrac{\eps}{2})$-invariant bounded set.
Finally, Proposition~\ref{prop:ISS_wrt_A_criterion} proves that $\Sigma$ is ISS w.r.t. $\A_{\eps,\gamma}$.
Since $\eps>0$ can be chosen arbitrarily, and since $\gamma \in\Kinf$, then $\gamma^{-1}(\tfrac{\eps}{2})$ can be made arbitrarily large by choosing sufficiently large $\eps$.
\end{proof}



\subsection{Semilinear evolution equations}
\label{sec:Semilinear_equations}

Here we specify the results obtained previously to semilinear evolution equations in Banach spaces.

Let $X$ be a Banach space and $A$ be the generator of a strongly continuous semigroup $T$ of bounded linear operators on $X$ and let $f:X\times U \to X$. Consider the system
\begin{equation}
\label{InfiniteDim}
\dot{x}(t)=Ax(t)+f(x(t),u(t)), \quad x(t) \in X,\ u(t) \in U.
\end{equation}

We study mild solutions of \eqref{InfiniteDim}, i.e. solutions $x:[0,\tau] \to X$ of the integral equation
\begin{align}
\label{InfiniteDim_Integral_Form}
x(t)=T(t)x(0) + \int_0^t T(t-s)f(x(s),u(s))ds,
\end{align}
belonging to the space of continuous functions $C([0,\tau],X)$ for some $\tau>0$.

We assume that the set of input values $U$ is a normed
linear space and that the input functions belong to the space
$\Uc:=PC(\R_+,U)$ of globally bounded, piecewise continuous functions $u:\R_+ \to U$, which
are right continuous. The norm of $u \in \Uc$ is given by
$\|u\|_{\Uc}:=\sup_{t \geq 0}\|u(t)\|_U$. 

{
We assume that the solution of \eqref{InfiniteDim} exists and is unique on $\R_+$ (i.e. \eqref{InfiniteDim} is forward complete). 
ISpS of \eqref{InfiniteDim} can be characterized as follows:
}
\begin{proposition}
\label{prop:ISpS_for_Evolution_Equations_in_Banach_Spaces}
Consider a BRS system \eqref{InfiniteDim} satisfying 
\begin{itemize}
    \item { $f:X \times U \to X$ is Lipschitz continuous on bounded subsets of $X$.}
    \item $f(x,\cdot)$ is continuous for all $x \in X$.
\end{itemize}
The following statements are equivalent:
\begin{enumerate}[label=(\roman*)]
    \item\label{enum:ISpS_Banach_item1} {\eqref{InfiniteDim}} is ISpS
    \item\label{enum:ISpS_Banach_item2} For any $s>0$ there is a bounded $s$-invariant set $\A \subset X$: {\eqref{InfiniteDim}} is ISS w.r.t. $\A$.
    \item\label{enum:ISpS_Banach_item5} There is a bounded set $\A \subset X$: {\eqref{InfiniteDim}} is ULIM w.r.t. $\A$.
\end{enumerate}
\end{proposition}

\begin{proof}
According to \cite[Section VII]{MiW17b}, a BRS system \eqref{InfiniteDim} satisfying assumptions of the proposition has a flow which is Lipschitz continuous on compact intervals {
(in \cite[Section VII]{MiW17b} a stronger assumption on $f$ has been imposed, but the assertion which we need here can be shown without these further requirements).
}
The input space $\Uc:=PC(\R_+,U)$ satisfies the assumptions of Lemma~\ref{lem:nonzero_Invariance_Aeps}.
Now the application of Theorems~\ref{thm:RFC_weak_attr_imply_pUGAS},~\ref{thm:ISpS_equivalent_ISS_wrt_bounded_sets} proves the claim of the proposition.
\end{proof}
%
%
%

\subsection{Ordinary differential equations}
\label{sec:ISpS_ODEs}

Finally, we turn our attention to the ISS theory of ODEs
\begin{eqnarray}
\dot{x}=f(x,u),
\label{eq:ODE_Sys}
\end{eqnarray} 
where $f:\R^n\times \R^m \to \R^n$ is locally Lipschitz continuous w.r.t. the first argument and inputs $u$ belong to the set $\Uc:=L_\infty(\R_+,\R^m)$ of Lebesgue measurable globally essentially bounded functions with values in $\R^m$.

{We show that for this class of systems the characterizations of ISpS developed in the previous sections can be considerably strengthened.
}

The following result has been shown in \cite{MiW17b} for $\A=\{0\}$ on the basis of \cite[Corollary III.3]{SoW96}. The proof for general bounded $\A$ is analogous and thus omitted.
\begin{proposition}
\label{prop:Approachability_equals_Uniform_Approachability}
Consider a system \eqref{eq:ODE_Sys} with $\Uc$ as above. Let $\A \subset \R^n$ be any bounded set. 
Then $\Sigma$ is ULIM w.r.t. $\A$ if and only if $\Sigma$ is LIM w.r.t. $\A$.
\end{proposition}


%

{
%

Sontag and Wang defined in \cite[Section VI]{SoW96} the following concept:
\begin{definition}
\label{def:compact-ISS} 
\eqref{eq:ODE_Sys} is called compact ISS, if there is a compact 0-invariant set $\A\subset \R^n$  so that \eqref{eq:ODE_Sys} is UAG w.r.t. $\A$.
\end{definition}

 }

{ Criteria for practical ISS of a system \eqref{eq:ODE_Sys} take a particularly simple form:}
\begin{corollary}
\label{cor:ODE_ISpS}
Let \eqref{eq:ODE_Sys} be forward complete.
The following statements are equivalent:
\begin{enumerate}[label=(\roman*)]
    \item\label{enum:ISpS_ODEs_item1} \eqref{eq:ODE_Sys} is ISpS
    \item\label{enum:ISpS_ODEs_item11} For any $s>0$ there is a compact $s$-invariant set $\A\subset \R^n$: \eqref{eq:ODE_Sys} is ISS w.r.t. $\A$.
    \item\label{enum:ISpS_ODEs_item2} { \eqref{eq:ODE_Sys} is compact-ISS.}
    \item\label{enum:ISpS_ODEs_item3} There is a bounded set $\A\subset \R^n$: \eqref{eq:ODE_Sys} is LIM w.r.t. $\A$.
\end{enumerate}
\end{corollary}

\begin{proof}
According to \cite[Proposition 5.1]{LSW96}, for \eqref{eq:ODE_Sys} forward completeness is equivalent to the boundedness of reachability sets.

\ref{enum:ISpS_ODEs_item1} $\Rightarrow$ \ref{enum:ISpS_ODEs_item11}. 
{By \cite[Proposition 5.5]{LSW96}, } the flow of \eqref{eq:ODE_Sys} is Lipschitz continuous on compact subsets. The input space $\Uc:=L_\infty(\R_+,\R^m)$ satisfies the assumptions of Lemma~\ref{lem:nonzero_Invariance_Aeps}.
Application of Theorem~\ref{thm:ISpS_equivalent_ISS_wrt_bounded_sets} proves that
for any $s>0$ there is a bounded $s$-invariant set $\A\subset \R^n$: \eqref{eq:ODE_Sys} is ISS w.r.t. $\A$.

Since $f$ is Lipschitz continuous w.r.t. the first argument, the solutions of \eqref{eq:ODE_Sys} depend continuously on initial data. Hence $\overline{\A}$ is again $s$-invariant (as a closure of an $s$-invariant set). Since $\A$ is bounded, $\overline{\A}$ is compact.

\ref{enum:ISpS_ODEs_item11} $\Rightarrow$ \ref{enum:ISpS_ODEs_item2} $\Rightarrow$ \ref{enum:ISpS_ODEs_item3}. Clear.

\ref{enum:ISpS_ODEs_item3} $\Rightarrow$ \ref{enum:ISpS_ODEs_item1}. Follows by 
Proposition~\ref{prop:Approachability_equals_Uniform_Approachability} and Theorem~\ref{thm:RFC_weak_attr_imply_pUGAS}.
%
\end{proof}

%

\section{Conclusion}

We have proved criteria for practical ISS for a general class of infinite-dimensional systems in terms of uniform limit property and in terms of ISS. The results are new already for the special case of ODE systems, and they strengthen substantially existing criteria for ISpS of ODEs.

\bibliographystyle{IEEEtran}



\bibliography{C:/GoogleDrive/TEX_Data/Mir_LitList}

%
%

\end{document}